\documentclass[11pt]{article}
\usepackage{amsmath,amsfonts,amsthm,amsrefs}

\usepackage{graphicx,color} %,pict2e}

\definecolor{refkey}{rgb}{1, 0.5, 0}  
\definecolor{labelkey}{rgb}{1,1,1}
\definecolor{labelkey}{rgb}{0.1,1,0.2}

\newtheorem{theorem}{Theorem}[section]
\newtheorem{lemma}{Lemma}[section]
\newtheorem{remark}{Remark}[section]

\newtheorem{definition}{Definition}[section]

\def\bel{\begin{equation}\label}
\def\eeq{\end{equation}}
\def\bega{\begin{array}}
\def\enda{\end{array}}

\def\ve{\varepsilon}

\title{Traveling Waves for a Microscopic Model of Traffic Flow}
\author{Wen Shen and Karim Shikh-Khalil \\ Mathematics Department, Penn State University, U.S.A.\\
wxs27{@}psu.edu, krs5562{@}psu.edu}

\begin{document}

%\pagewiselinenumbers\linenumbers
%\modulolinenumbers[1]

\maketitle

\begin{abstract}
We consider the follow-the-leader model for traffic flow.
The position of each car $z_i(t)$ satisfies an ordinary differential equation, 
whose speed depends only on the relative position  $z_{i+1}(t)$ of the car ahead. 
Each car perceives a local density $\rho_i(t)$. 
We study a discrete traveling wave profile $W(x)$ along which the trajectory 
$(\rho_i(t),z_i(t))$ traces such that $W(z_i(t))=\rho_i(t)$ for all $i$ and $t>0$;
see definition~\ref{def:2}. 
We derive a delay differential equation satisfied by such profiles.
Existence and uniqueness of solutions are proved, 
for the two-point boundary value problem where the car densities
at $x\to\pm\infty$ are given. 
Furthermore, we show that such profiles are locally stable, 
attracting nearby monotone  solutions of the follow-the-leader model. 
\end{abstract}

MSC2010:  35L02, 35L65, 34B99, 35Q99

%%%%%%%%%%%%%%%%%%%%%%%%
\section{Introduction and Preliminaries}
\setcounter{equation}{0}

We consider a microscopic model for traffic flow.
Let $\ell$ be the length of all the cars, and let 
$z_i (t)$ be the position of $i$th car at time $t$.  
We order the indices for the cars such that 
\begin{equation}\label{zi}
z_i(t) \le z_{i+1}(t)-\ell\qquad \mbox{for every}~i\in\mathbb{Z}.
\end{equation}
For a car with index $i$, we define the local density
\begin{equation}\label{defrho}
\rho_i(t)  \; \dot= \; \frac{\ell}{z_{i+1}(t)-z_i(t)}. 
\end{equation}
Note that if $\rho_i= 1$, then the two cars with indices $i$ and $i+1$ will be
bumper-to-bumper.  Thus $0 \le \rho_i \le 1$ for all $i$. 

We assume that the speed of the car with index $i$ depends solely on 
the local density $\rho_i$, i.e., 
\begin{equation}\label{FtL}
 \dot z_i(t)  = V \cdot \phi(\rho_i(t)) 
 = V \cdot \phi\left(\frac{\ell}{z_{i+1}(t)-z_i(t)}\right).
\end{equation}
Here $V$  is the speed limit, 
and the function
$\phi(\rho)$, defined on $\rho\in[0,1]$,  satisfies
\begin{equation}\label{phiP} 
\phi(1)=0, \qquad \phi(0)=1, \qquad  \phi'(\rho)\le - \hat c_0<0\quad \mbox{for all}~\rho\in[0,1].
\end{equation}
We remark that a popular choice for $\phi(\cdot)$  is the 
Lighthill-Whitham model~\cite{MR0072606}:
\begin{equation}\label{phi}
\phi(\rho)= 1-\rho.
\end{equation}
Given an initial distribution of car positions $\{z_i(0)\}$, 
the system~\eqref{FtL} depicts the ``follow-the-leader'' 
behavior of each car. We refer to this model as the FtL model. 

Note that~\eqref{FtL} can be rewritten as a system of ODEs for the
discrete density functions $\rho_i(t)$,
\begin{equation}
\label{rhodot}
    \dot \rho_i(t) 
=-\frac{\ell \left( \dot z_{i+1} - \dot z_i \right) }{(z_{i+1}-z_i)^2} 
= \frac{V}{\ell} \rho_i^2 \cdot \big( \phi(\rho_{i}) - \phi(\rho_{i+1})\big).
\end{equation}
If one uses~\eqref{phi}, then \eqref{rhodot} becomes
\begin{equation}\label{rhodotx}
    \dot \rho_i(t) 
= \frac{V}{\ell} \rho_i^2 \left( \rho_{i+1} - \rho_{i}\right).
\end{equation}

Given the initial positions of the cars $z_i(0)$ and the speed of the leader
as $i\to\infty$,
the existence of solution for the ODE system~\eqref{FtL}
is established in the literature~\cite{MR3356989, HoldenRisebro}. 
We can define a piecewise constant function $\rho^{\ell}(t,x)$
from the discrete densities $\{\rho_i\}$ as
\begin{equation}\label{disrho}
\rho^{\ell}(t,x)\,\dot=\,\rho_i(t), \quad  \mbox{for}~x\in[z_i(t), z_{i+1}(t)).
\end{equation}
As $\ell\to 0$ and the number of the cars tends to $\infty$,
under suitable assumptions one has 
the convergence $\rho^\ell(t,x) \to \rho(t,x)$, 
where the limit function $\rho(t,x)$ provides a weak solution for the scalar
conservation law
\begin{equation}\label{PDE}
\rho_t + f(\rho)_x =0, \qquad \mbox{where}\quad 
f(\rho) = V \rho \cdot \phi(\rho) .
\end{equation}
See~\cite{MR3356989} for a proof using direct properties of the
solutions of~\eqref{FtL},
and some more recent works~\cite{HoldenRisebro,HoldenRisebro2} where
the same results are achieved
utilizing  a Lagrangian formulation and the properties of monotone 
numerical schemes. 
For other related works including model derivations, 
analysis, and treatment of various 
conditions, we refer to~\cite{MR2727134,MR1952890, MR1952895,MR3253235,MR2006201,
MR3217759,MR3541527,MR3605557,MR3177735} and the references therein.

It is well-known that, in the solutions of the nonlinear conservation law~\eqref{PDE}, 
discontinuities can form in finite time even with smooth initial data.
Such discontinuities are known as shocks. 
In the literature for traffic flow, the flux $f(\rho)$ is typically concave  such that
\begin{equation}\label{fP}
f(0)=f(1)=0,\qquad f''(\rho) \le -c_0 < 0, \qquad  f'(\rho^*)=0 
\qquad \mbox{for a unique}~\rho^*\in(0,1).
\end{equation}
Note that~\eqref{fP} 
holds with~\eqref{phi}, where 
$f(\rho)=V \rho (1-\rho)$.
In general, 
\eqref{fP}  leads to additional assumptions on $\phi(\rho)$, besides~\eqref{phiP}, i.e.,
\begin{equation}\label{phi2}
\phi''(\rho) \le  -\frac{1}{\rho}\; \left[ 2 \phi'(\rho)+ c_0/V\right],
\qquad \rho\in[0,1].  
\end{equation}

Since $f''<0$, only upward jumps are admissible in the solutions of~\eqref{PDE}.
The solution for the Riemann problem with initial data 
\[
\rho(0,x) = \begin{cases}
\rho_- , & x<0, \\ \rho_+, & x>0,
\end{cases} \qquad \mbox{and}\quad \rho_- < \rho_+
\]
results in a single shock which travels 
with the Rankine-Hugoniot jump speed:
\[ 
\sigma=\frac{f(\rho_-)-f(\rho_+)}{\rho_- - \rho_+}.
\]
The shock is stationary when $f(\rho_-)=f(\rho_+)$.

In this work we seek a ``discrete traveling wave profile'' for the FtL model, 
as a corresponding approximation to the shock waves for the conservation 
law~\eqref{PDE}. 
To fix the idea, 
we start with a monotone stationary profile $W(x)$ such that the position of the 
point $(z_i(t), \rho_i(t))$ traces along the graph of the function $W(x)$ 
as time $t$ evolves. To be precise, we require
\begin{equation}\label{W1}
W(z_i(t)) = \rho_i(t), \qquad \forall t\ge 0, \quad \forall i.
\end{equation}
We remark that for general traveling waves with speed $\sigma\not=0$, 
the profile will be stationary in the shifted coordinate $x \mapsto \xi=x-\sigma t$,
see the discussion in Section~5. 

Differentiating both sides of~\eqref{W1} in $t$, 
and using~\eqref{FtL} and~\eqref{rhodot},
one gets
\[ 
W'(z_i) = \frac{\dot \rho_i}{\dot z_i} = 
\frac{\rho_i^2}{\ell \cdot \phi(\rho_i)} \Big[\phi(\rho_i)-\phi(\rho_{i+1})\Big]
=
\frac{W^2(z_i)}{\ell \cdot \phi(W(z_i))} \Big[\phi(W(z_i))-\phi(W(z_{i+1}))\Big].
\]
Note that 
\[
 z_{i+1}= z_i + \frac{\ell}{\rho_i} = z_i + \frac{\ell}{W(z_i)}.
 \]

Since $z_i$ is randomly chosen, we write $x$ for $z_i$, 
and obtain the following equation:
\begin{equation}\label{eq:W}
W'(x) = \frac{W^2(x)}{\ell \cdot \phi(W(x))} \cdot
\left[ \phi(W(x)) -\phi(W(x+\frac{\ell}{W(x)})) \right].
\end{equation}
If $W(x)=0$ for some $x$, then we set 
\[ W(x+\frac{\ell}{W(x)}) = W(+\infty) .\]
Equation~\eqref{eq:W} is a Delay Differential Equation (DDE). 
Furthermore, \eqref{eq:W} is autonomous since  
the righthand side does not depend on $x$ explicitly. 

Once the ``initial data'' is given on an interval $[\hat x,\infty)$ 
for any $\hat x$,  the DDE~\eqref{eq:W} can be solved backwards in $x$,
and the profile $W(x)$ can be obtained for all $x\le \hat x$.
This is in agreement with the following-the-leader principle. 

In this paper we study in detail the DDE~\eqref{eq:W}.
In particular, we study the ``two-point-boundary-value" problem. 
To be specific, 
we seek solutions of~\eqref{eq:W} that satisfies the boundary conditions at the infinities:
\begin{equation}\label{BC}
\lim_{x\to\pm\infty} W(x)=\rho_\pm, \qquad  0 \le \rho_- \le  \rho_+\le 1.
\end{equation}
In the case of the stationary profile $W(x)$,  $\rho_\pm$ must further satisfy
\begin{equation}\label{BCcon}
 f(\rho_-)=f(\rho_+) \;\dot= \;\bar f, \qquad 0\le \rho_- \le \rho^*\le \rho_+\le 1
 \quad \mbox{where}\quad
 f'(\rho^*)=0.
\end{equation}

Note that any horizontal shift of the profile $W(x)$ is again a profile. 
Thus, a unique profile can be achieved by requiring a ``location-fixing'' condition
at $x=0$, say 
\[
W(0)=\rho^*.
\]
We show that, for any given $\rho_\pm$ satisfying~\eqref{BCcon},
there exists a profile $W(x)$, unique up to horizontal shifts. 
Furthermore, such traveling waves are local attractors
for the solutions of the FtL model~\eqref{FtL}. 

In the literature, solutions for the conservation law~\eqref{PDE} 
are approximated by various approaches.
These include 
the viscous equations, kinetic models with relaxation terms,
and various numerical approximations. 
For many of the approximate solutions, the study of traveling wave profiles
is one of the key techniques in the analysis. 
In this paper, we consider the microscopic ``particle" model and its traveling waves, 
filling a missing piece in the literature.

We mention also a study on traveling waves for a non-standard 
 integro-differential equation
modeling  slow erosion~\cite{MR3115842}, where uniqueness and local stability
are achieved. 

The rest of the paper is organized as follows. 
For stationary profiles $W(x)$, 
in Section 2 we prove several technical Lemmas.
These results are utilized in Section 3 where we prove the existence 
and uniqueness of the profile. 
Furthermore, such profiles are local attractors 
for solutions of the FtL model~\eqref{FtL}, proved in Section 4. 
Extension to general traveling waves with non-zeros speed
is outlined in Section 5. 
Finally, concluding remarks and further open problems are discussed 
in Section 6.

%%%%%%%%%%%%%%%%%%%%%%
\section{Technical Lemmas; Properties of the Stationary Profile}
\setcounter{equation}{0}

We consider the stationary profile $W(x)$,
satisfying the DDE~\eqref{eq:W} with boundary 
conditions~\eqref{BC}-\eqref{BCcon}.

We first provide a formal argument which makes connection 
between the profile $W(x)$ and the viscous shock for the conservation law~\eqref{PDE}.
Assuming that 
$\ell/W(x)>0$ is very small,
by Taylor expansion we have
\begin{equation}\label{Taylor}
 \phi\left(W(x+\frac{\ell}{W(x)})\right)-\phi(W(x)) 
= \frac{\ell}{W} (\phi(W))_x + \frac12 \left(\frac{\ell}{W}\right)^2 
(\phi(W))_{xx} + \mathcal{O}\left(\left(\frac{\ell}{W}\right)^3\right).
\end{equation}
Dropping the higher order terms, 
the DDE~\eqref{eq:W} is  approximated by 
\[
W_x = \frac{W}{\phi(W)} \cdot \left[  - (\phi(W))_x - \frac{\ell}{2W} (\phi(W))_{xx}
\right]
\]
This equation can be manipulated into:
\[
\phi(W) W_x + W (\phi(W))_x=-\frac{\ell}{2} (\phi(W))_{xx},
\]
and then
\[
(W \cdot\phi(W))_x=  \frac{1}{V}f(W)_x=-\frac{\ell}{2} (\phi(W))_{xx}.
\]
We conclude
\begin{equation}\label{a2}
f(W)_x = \left( - \frac{V\ell}{2} \phi'(W) W_x \right) _x 
\end{equation}

Now we consider the  viscous conservation law
\[ \rho_t + (f(\rho))_x = \ve \rho_{xx}.\]
Stationary viscous shock waves $\bar \rho(x)$ must satisfy the ODE
\begin{equation}\label{a3}
 f(\bar\rho)_x = \ve \bar\rho_{xx} = \left(\ve \bar\rho_{x} \right)_x.
 \end{equation}
 
We observe that the ODEs \eqref{a2} and \eqref{a3} are connected through the relation:
\[
\ve \approx -\frac{V\ell}{2} \phi'(W(x)) .
\]
For the case  $\phi(\rho)=1-\rho$ where $\phi'(\rho)=-1$, we have the connection
$ \ve \approx \frac{V\ell}{2}$.

\bigskip

For any given profile $W(x)$, one can generate a distribution of car positions $\{z_i\}$,
and vise versa.
We make the following definitions. 

\begin{definition}\label{def:1}
Let the function $x\mapsto W\in(0,1]$ 
be given for $x\in\mathbb{R}$. 
We call a sequence of car positions  $\{z_i\}$ 
\textbf{a distribution generated by $W(x)$}, if
\begin{equation}\label{eq:def1}
z_{i+1}-z_i = \frac{\ell}{W(z_i)}, \qquad \forall i\in\mathbb{Z}.
\end{equation}
If one imposes $z_0=0$, then the distribution  is unique.
\end{definition}

\begin{definition}\label{def:2}
Given a profile $W(x)$ and a distribution of car positions $\{z_i(t)\}$.
Let  $\{\rho_i(t)\}$ be the corresponding discrete densities for the cars, 
computed as~\eqref{defrho}.
We say that $\{z_i(t)\}$ \textbf{traces along $W(x)$},  if 
\[
 W(z_i(t)) = \rho_i(t)=\frac{\ell}{z_{i+1}(t)-z_i(t)}, \qquad \forall i, t.
 \]
\end{definition}

The following Lemma is an immediate consequence of these definitions.

\begin{lemma}\label{lm:0}
Let $W(x)$ be a given profile and $\{z_i(0)\}$ be a distribution generated by $W(x)$.  
Let $\{z_i(t)\}$ be the solution of~\eqref{FtL} with initial data $\{z_i(0)\}$.
Then, $W(x)$ satisfies the DDE~\eqref{eq:W} 
if and only if $\{z_i(t)\}$ traces along $W(x)$.
\end{lemma}

Our first theorem states existence and uniqueness of monotone solutions of~\eqref{eq:W} 
as an initial value problem, under suitable assumptions on the initial data. 

\begin{theorem}\label{IVP}
Fix an $\hat x$. 
Let $\psi(x) \in (0,1)$  be a continuous monotone 
function defined on the interval $x\ge \hat x$
such that $\psi'(x) >0$ for all $x\ge \hat x$.
Let $W(x)$ be the solution of the DDE~\eqref{eq:W} on $x< \hat x$,
solved backwards in $x$,  with initial data $\psi(x)$ given on $x\ge\hat x$. 
Then, there exists a unique positive solution $W(x)$, which is 
monotone increasing such that
\begin{equation}\label{2.6}
W'(x) >0, \qquad W(x) >0, \qquad \mbox{for}~ x \le \hat x.
\end{equation}
\end{theorem}

\begin{proof}
The existence and uniqueness of the solution $W(x)$ 
for the initial value problem follows from an iteration argument. 
It is understood that the derivative $W'(x)$ in~\eqref{eq:W} is the left derivative. 
We clearly have
\[
W'(\hat x-) = \frac{\psi(\hat x)}{\phi(\psi(\hat x))} \cdot
\frac{\phi(\psi(\hat x)) - \phi(\psi(\hat x+\ell/\psi(\hat x)))}{\ell/\psi(\hat x)} >0.
\]
Now consider the interval $x\in I_1 = [\hat x -\ell, \hat x]$. 
We claim that, if $W(x)$ exists on $I_1$, then $W(x)\ge 0$. 
Indeed, the lower bound $W(x)\ge0$ is clear since  0 is a critical point. 
Assuming that $W(x)$ becomes negative 
on some subset of $I_1$, then there exists a point $x_0\in I_1$ such that
\[  W(x_0)=0, \quad W'(x_0) >0 .\]
But this is not possible because by~\eqref{eq:W} we have 
\[
W'(x_0) = \frac{W^2(x_0)}{\ell \phi(W(x_0)))} \left[ \phi(W(x_0)) - \phi(\rho_+)\right]=0, \quad \mbox{where}~\rho_+ =\lim_{x\to\infty}\psi(x), 
\]
a contradiction. 
 
We further claim that, if $W(x)$ exists on $I_1$, then it is 
monotonically increasing. 
We prove by contradiction. 
Assume that $W(x)$ is not monotone on $I_1$. 
Then there exists a value  
$\tilde x $, with $\tilde x < \hat x$,  where $W'$ changes sign, such that
\begin{equation}\label{mono1}
 W'(\tilde x)=0, \quad \mbox{and}\quad W'( x)>0, \quad x>\tilde x.
 \end{equation}
However, this would imply
\[
 W'(\tilde x) =\frac{ W^2(\tilde x)}{\ell \phi(W(\tilde x))} \cdot
\left[ \phi(W(\tilde x))-\phi\left(W(\tilde x + \frac{\ell}{W(\tilde x)})\right)\right]=0,
\]
thus
\[
 \phi\left(W(\tilde x)\right)-\phi\left(W\left(\tilde x + \frac{\ell}{W(\tilde x)}\right)\right)=0
\qquad \Rightarrow\qquad 
W\left(\tilde x\right)=W\left(\tilde x + \frac{\ell}{W(\tilde x)}\right), 
\]
a contradiction to~\eqref{mono1}. 

Thus, we deduce that 
\[
 x+\ell/W(x) > \hat x, \qquad \mbox{for every} ~ x \in I_1,
 \]
which means, 
\[
 W( x+\ell/W(x) ) = \psi( x+\ell/W(x) ) , \qquad \mbox{for every} ~ x \in I_1.
 \]
Then, the equation~\eqref{eq:W} reduces to an ODE of the form
\[
 W'(x) = \mathcal{F}(W, \psi(x+\ell/W))
 \;\dot=\;\frac{W^2(x)}{\ell \phi(W(x))} \left[ \phi(W(x)) - \phi(\psi(x+\ell/W(x)))\right] .
 \]
Since $\mathcal{F}$ is Lipschitz in both arguments, $\psi$ is continuous
and Lipschitz for $W>0$, 
by standard ODE theory, the solution $W(x)$ exists and is unique on $I_1$. 

One can the iterate the argument on 
the intervals $I_k = [\hat x - (k+1)\ell, \hat x - k\ell]$, 
for $k=1,2,\cdots$, completing the proof. 
\end{proof}

\begin{remark}
For  general references on  standard theory for 
delay differential equations, see~\cite{MR0141863, MR0477368}. 
We remark that our equation~\eqref{eq:W} does not fall into the standard setting,
therefore we provide a simple proof for existence and uniqueness of solutions. 
We further note that, if the initial condition shall be monotonically decreasing,
such that $\psi'(x) <0$ for $x>\hat x$,  
the global existence of solution $W(x)$ on $x\in(-\infty,\hat x]$ fails. 
One simply observes that $W'(x) <0$, so $W(x)$ increases as $x$ decreases,
and $W'(x) $ blows up to infinity as $W(x)$ approaches 1.
\end{remark}

Next Lemma describes the asymptotic behavior at the limits as $x\to\pm\infty$.

\begin{lemma}\label{lm:4}
(Asymptotic Limits.)
Assume that $W(x)$ is a solution of~\eqref{eq:W} that satisfies the boundary condition
\[
W(-\infty)=\rho_-, \qquad W(+\infty)=\rho_+.
\]
Then, we have the followings.
\begin{itemize}
\item As $x\to\infty$, 
$W(x)$ can approach $\rho_+$ at an exponential rate only if $\rho_+ > \rho^*$.
The exponential rate $\lambda_+$ satisfies the estimate
\begin{equation}\label{lambda+}
 \lambda_+ >  \frac{2\rho_+}{\ell} \cdot
 \ln \left(1-\frac{f'(\rho_+) \rho_+}{f(\rho_+)}\right).
\end{equation}
\item 
Similarly, as 
$x\to-\infty$, 
$W(x)$ can approach $\rho_-$ at an exponential rate only if $\rho_- < \rho^*$.
The exponential rate $\lambda_-$ satisfies the estimates
\begin{equation}\label{lambda-2}
-\frac{\rho_-}{\ell}\cdot \ln \left(1-\frac{f'(\rho_-) \rho_-}{f(\rho_-)}\right)
<\lambda_- <
-\frac{2 \rho_-}{\ell} \cdot \ln \left(1-\frac{f'(\rho_-) \rho_-}{f(\rho_-)}\right).
\end{equation}
\end{itemize}
\end{lemma}

\begin{proof} 
\textbf{Step 1.} 
Consider the asymptotic behavior as $x\to+\infty$.
By assumption we  have $W'(x) \to 0$ as  $x \to\infty$.
Now, for $x$ large, we write 
\[
W(x) =\rho_+ +\eta(x)
\]
where $\eta(x)$ is the first order perturbation.
Plugging this into~\eqref{eq:W}, and neglecting the higher order terms,
we obtain the following linearized equation for $\eta(x)$: 
\begin{equation}\label{eta}
\eta' (x) = - \phi'(\rho_+)  \cdot \frac{\rho_+^2}{\ell \cdot \phi(\rho_+)}  
\Big[ \eta(x+\ell/\rho_+) - \eta(x) \Big].
\end{equation}
Denoting the positive constants as
\begin{equation}\label{defab}
a \;\dot=\;\frac{\ell}{\rho_+}, 
\qquad 
b\; \dot=\; - \phi'(\rho_+)  \cdot \frac{\rho_+}{ \phi(\rho_+)}, 
\end{equation}
we can write
\begin{equation}\label{eta2} 
\eta'(x) = \frac{b}{a}\cdot(\eta(x+a)-\eta(x)).
\end{equation}
This is a linear delay differential equation, 
which can be solved explicitly using the characteristic equation.
Seeking solution of the form  
\begin{equation}\label{eq:eta}
\eta(x) =M e^{-\lambda x}
\end{equation} 
where $M$ is an arbitrary constant (which could be both negative or positive),
the rate $\lambda$ satisfies the characteristic equation
\begin{equation}\label{eq:Gdef}
G(\lambda)\;\dot=\;b \cdot \left( e^{-a \lambda} -1\right)+a \lambda=0.
\end{equation}
We locate all the zeros for the function $G(\lambda)$, 
in particular the positive ones. 
We observe
\begin{eqnarray}
& &G(0)=0,\mbox{}\label{G}\\
&&G'(\lambda) = -ab\; e^{-a\lambda} +a,   \qquad\qquad  G'(0)=a(1-b), \label{DG}\\
&& G''(\lambda)=a^2b\; e^{-a\lambda}>0 , \qquad\qquad  G''(0)=a^2b,\label{DDG}\\
&& G'''(\lambda) = -a^3 b\; e^{-a \lambda} <0. \label{DDDG}
\end{eqnarray}
Thus, $G(\cdot)$ is a convex function which goes through the origin. 
Typical graphs of $G(\lambda)$ for different  values of $b$
are illustrated in Figure~\ref{fig:G}. 
We have:
\begin{itemize}
\item
If $b=1$, then there is only one zero $\lambda_0=0$.
\item
If $b<1$, then there are two zeros $\{\lambda_0,\lambda_-\}$, where 
 $\lambda_- <0$.   
\item
If $b>1$, then there are two zeros $\{\lambda_0,\lambda_+\}$, 
where $\lambda_+ >0$.
\end{itemize}

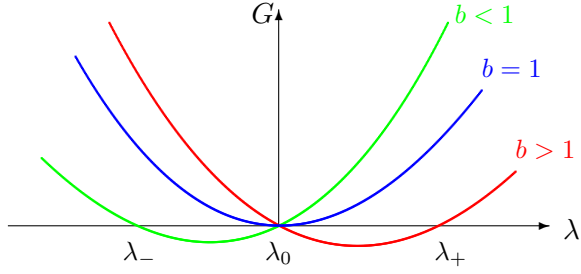
\begin{figure}[htbp]
\begin{center}
\setlength{\unitlength}{0.9mm}
\begin{picture}(80,40)(-40,-6)  % -- Left ---
\put(-40,0){\vector(1,0){80}}\put(42,-2){$\lambda$}
\put(0,0){\vector(0,1){32}}\put(-4,30){$G$}
\put(-2,-5){\small $\lambda_0$}
\put(23,-5){\small $\lambda_+$}
\put(-23,-5){\small $\lambda_-$}
\thicklines
\color{red}\qbezier(-25,30)(2,-22)(35,8)\put(35,10){\small$b>1$}
\color{green}\qbezier(-35,10)(-1,-22.5)(25,30)\put(26,30){\small$b<1$}
\color{blue}\qbezier(-30,25)(-3,-22.3)(30,20)\put(30,22){\small$b=1$}
\end{picture}
\caption{Typical graphs of $G(\lambda)$ and location of the zeros.}
\label{fig:G}
\end{center}
\end{figure}

Recall that 
\[
f(\rho)=V\rho\cdot\phi(\rho), \qquad f'(\rho^*)=0,\qquad f''(\rho) <0, 
\]
we have 
\[
 b = \frac{-\phi'(\rho_+) \rho_+}{\phi(\rho_+)}
=1-\frac{f'(\rho_+) \rho_+}{f(\rho_+)}~:~
\begin{cases}  =1, \qquad & \mbox{if}\quad \rho_+ = \rho^* ,\\
>1, & \mbox{if}\quad  \rho_+ >\rho^* ,\\
<1, & \mbox{if}\quad \rho_+ <\rho^* .
\end{cases}
\]

Thus we conclude:
\begin{itemize}
\item 
If $\rho_+=\rho^*$, then $\lambda=0$ and we have the trivial solution 
$\eta(x)\equiv0$, and thus $W(x)\equiv \rho^*$.
\item 
If $\rho_+<\rho^*$, then the other rate $\lambda_-<0$ indicates exponential
growth of $\eta(x)$, which is not valid.
The only possible solution is the trivial one.
\item
If $\rho_+>\rho^*$, then the other zero $\lambda_+>0$ indicates exponential
decay of $\eta(x)$ in the limit as $x\to\infty$. This is the valid case. 
\end{itemize}
Therefore, if $\rho_+>\rho^*$,  we can have the asymptotic limit as $x\to\infty$, 
\[
 W(x) ~\to~ \rho_+ + M e^{-\lambda_+ x} .
 \]

Finally, we derive an estimate on the rate $\lambda_+$. Let 
\[
\lambda_*= \frac{1}{a} \ln (b), \qquad\mbox{where}\quad  G'(\lambda_*)=0.
\]
By the properties~\eqref{G}-\eqref{DDDG}, we conclude the estimate 
\[
 \lambda_+ > 2 \lambda_* = \frac{2}{a} \ln (b) ,
\]
proving~\eqref{lambda+}.

\bigskip
\textbf{Step 2.} 
A similar computation can be carried out for  $x\to-\infty$. 
We write 
\[
 W(x) = \rho_- + \zeta(x)
 \]
where $\zeta(x)$ is a small perturbation.  
The linearized equation for $\zeta(x)$  becomes
\[
 \zeta'(x) = - \frac{\phi'(\rho_-) \rho_-^2}{\ell \phi(\rho_-)} \Big[ 
\zeta(x+\ell/\rho_-) -\zeta(x) \Big].
\]
Denoting the positive constants
\begin{equation}\label{hab}
 \hat a \;\dot=\; \ell/\rho_-,
\qquad
\hat b \; \dot= \; -\frac{\phi'(\rho_-) \rho_-}{\phi(\rho_-)}
=1-\frac{f'(\rho_-) \rho_-}{f(\rho_-)},
\end{equation}
and seeking solutions of the form
\begin{equation}\label{eq:zeta}
 \zeta(x) = \hat M e^{\lambda x},
 \end{equation}
we arrive at the characteristic equation
\[
 H(\lambda)\; \dot= \; \hat b(e^{\hat a \lambda} -1) - \hat a \lambda =0.
 \]
To seek positive zeros of $H(\cdot)$, we first observe that $H(0)=0$. 
Furthermore, we have
\begin{eqnarray*}
H'(\lambda)=\hat a \hat b e^{\hat a \lambda} -\hat a, &\quad&
H'(0)= \hat a (\hat b-1), \\
H''(\lambda)= \hat a^2\hat b e^{\hat a \lambda}>0, 
&\quad &
H'''(\lambda) = \hat a^3\hat b e^{\hat a \lambda}>0.
\end{eqnarray*}
Thus, positive rate $\lambda=\lambda_-$ 
exists only for the case when $\hat b <1$, 
i.e., when $\rho_- < \rho^*$.

A similar computation as for~\eqref{lambda+} leads to an estimate
with both upper and lower bounds:
\begin{equation}\label{lambda-}
 -\frac{1}{\hat a} \ln (\hat b)  <\lambda_- <  -\frac{2}{\hat a} \ln (\hat b),
\end{equation}
proving~\eqref{lambda-2}
\end{proof}

\begin{remark}\label{rk:2.1}
Lemma~\ref{lm:4} indicates that the boundary conditions at $x\to\pm\infty$
are valid only when $\rho_- \le \rho^* \le \rho_+$. 
Together with Theorem~\ref{IVP}, we conclude that stationary profiles of $W(x)$,
if they exist,
are monotonically increasing.
This corresponds to the upward jumps of the admissible shocks 
for the conservation law~\eqref{PDE}.

Consequently, the constant $M$ in~\eqref{eq:eta} is now negative, and 
$\hat M$ in~\eqref{eq:zeta} is positive. 
\end{remark}

\begin{remark}\label{rk:2.2}
If $\rho_-=\rho_+=\rho^*$, then $\lambda=0$, and 
one has the trivial solution  $W(x) \equiv \rho^*$. 
\end{remark}

\begin{remark}\label{rk:2.3}
The estimates~\eqref{lambda+}-\eqref{lambda-2} can be expressed in different ways. 
Indeed,  \eqref{lambda+} implies
\begin{equation}\label{lambda++}
\lambda_+ > \frac{2 \rho_+}{\ell} \cdot 
\ln\left(1+ \frac{c_0\rho^*}{f(\rho_+)} (\rho_+-\rho^*)\right),
\end{equation}
where $-c_0$ is an upper bound for $f''$, see~\eqref{fP}.
Similarly, if $\rho_->0$ is close to $\rho^*$, the estimate~\eqref{lambda-2} implies
\begin{equation}\label{lambda--}
-\frac{\rho_-}{\ell} \cdot \ln\left(1-\frac{c_0\rho^*}{f(\rho^*)} (\rho^*-\rho_-)\right)
< \lambda_-  <
-\frac{2\rho_-}{\ell} \cdot \ln\left(1-\frac{c_0\rho^*}{f(\rho^*)} (\rho^*-\rho_-)\right),
\end{equation}
On the other hand, if $\rho_-$ is close to 0,  we have a different estimate
\begin{equation}\label{lambda---}
-\frac{\rho_-}{\ell} \cdot \ln\left(\hat c_0 \rho_-\right)
< \lambda_-  <
-\frac{2\rho_-}{\ell} \cdot \ln\left(\hat c_0 \rho_-\right),
\end{equation}
where $-\hat c_0$ is the upper bound for $\phi'$, see~\eqref{phiP}. 
\end{remark}

\begin{remark}\label{rk:2.4}
By the estimate~\eqref{lambda++}, as $\rho_+ \to 1$ we have
$\lambda_+ \to +\infty$.  
This indicates that if $\rho_+=1$,  $W(x)$ approaches $\rho_+$ instantly. 
Thus, for some $\hat x$, we must have $W(x) =1$ for $x \ge \hat x$.

On the other end, by estimate~\eqref{lambda---}, 
as $\rho_- \to 0$, we have  $\lambda_- \to \infty$. 
Thus, if $\rho_-=0$, we must have $W(x)=0$ for $x<\hat x$, for some $\hat x$.

One  concludes that, if 
\[
\rho_-=0, \qquad \rho_+=1,
\]
the only possible profile is a step function with the jump located at some $\hat x$.
This represents the scenario where cars are bumper-to-bumper on $x>\hat x$,
and the road is empty for $x<\hat x$. 
\end{remark}

The next lemma is most interesting.
It shows that, if $W(x)$ is a stationary monotone profile such 
that the solutions $\{z_i(t)\}$  
of~\eqref{FtL} traces along, then the distribution $\{z_i(t)\}$
demonstrates a ``periodic'' pattern. 

\begin{lemma}\label{lm:2} (Periodicity.)
Let $W(x)$ be a monotone profile, and let $\{z_i(0)\}$ 
be the initial positions of cars  generated by the  profile $W(x)$.
Let $\{z_i(t)\}$ be the solution of the FtL model~\eqref{FtL} with this initial data.
Then, the followings are equivalent. 
\begin{itemize}
\item[(E1)]  
$W(x)$ satisfies the DDE~\eqref{eq:W}.

\item[(E2)]  
The solutions $\{z_i(t)\}$ exhibit the following periodic behavior. 
There exist a ``period'' $t_p$, 
such that after the period each car takes over the initial position of its leader, i.e.,
\begin{equation}\label{per}
z_i(t+t_p) = z_{i+1}(t), \qquad \forall i, \quad \forall t \ge 0.
\end{equation}
\end{itemize}
\end{lemma}

\begin{proof} 
\textbf{Step 1.} 
We first prove that (E2) $\Rightarrow$ (E1). 
Without loss of generality, we consider a car initially located at 
$z_0(0)=x$ for some $x$,  
and its leader, initially located at 
\[
z_1(0)=x + \frac{\ell}{W(x)}.
\]
Thus, the evolution of $z_0(t)$ satisfies the ODE
\[
 \frac{d z_0}{dt} = V \phi(W(z_0)), \qquad z_0(0)=x, \quad z_0(t_p)=z_1(0)=x+\ell/W(x).
 \]
This is a separable equation.
(E2) implies the following identity
\begin{equation}\label{GI0}
 \int_x^{x+\ell/W(x)}   \frac{1}{V \; \phi(W(z))}\, dz = \int_0^{t_p}dt
=  t_p=\mbox{constant},\qquad \forall x.
\end{equation}
Differentiating~\eqref{GI0} in $x$, we deduce
\begin{equation}\label{GI2}
\left(1-\frac{\ell W'(x)}{W(x)^2}\right) \cdot\frac{1}{\phi(W(x+\ell/W(x)))} - \frac{1}{\phi(W(x))} =0,
\end{equation}
which easily leads to~\eqref{eq:W},  proving (E1).

\textbf{Step 2.}
To prove the indication  (E1) $\Rightarrow$ (E2),
assume that $W(x)$ satisfies the DDE~\eqref{eq:W}. For a given time $t$,
let $\{z_i(t)\}$ be a distribution of cars generated by $W(x)$, 
as in Definition~\ref{def:1}. 
We write now
\[
x=z_i(t), \qquad z_{i+1}(t)=x + \frac{\ell}{W(x)}.
\]
Since $W(x)$ solves  \eqref{eq:W}, it satisfies \eqref{GI2}. 
The time it takes for the $i$th car to reach the original position of its leader $z_{i+1}$ 
is
\begin{equation}
t_p(x) = \int_x^{x+\ell/W(x)}   \frac{1}{V \; \phi(W(z))}\, dz.
\end{equation}
By~\eqref{GI2} we immediately deduce that $t_p'(x)=0$ for all $x$, 
thus $t_p(x) \equiv $ constant. 
\end{proof}

\medskip

The next Lemma connects the period $t_p$ to the limit values of $W(x)$
at $x\to\pm\infty$.

\begin{lemma}\label{lm:3}
Let $W(x)$ and $\{z_i(t)\}$ be given as in the setting of Lemma~\ref{lm:2}. 
Let $\rho_-, \rho_+$ be two states that satisfy
\begin{equation}\label{BC2}
f(\rho_-)=f(\rho_+)=\bar f, \qquad 
\rho_- \le \rho^*\le \rho_+.
\end{equation}
Then, the following additional properties are equivalent.
\begin{itemize}
\item[(E3)] 
$W(x)$ 
satisfies the boundary conditions
\begin{equation}\label{BC3}
\lim_{x\to\pm\infty} W(x)=\rho_\pm.
\end{equation}
\item[(E4)]
The period $t_p$ for $\{z_i(t)\}$ is given as
\begin{equation}\label{tp}
t_p = \frac{\ell}{\bar f} = \frac{\ell}{f(\rho_-)} = \frac{\ell}{f(\rho_+)}.
\end{equation}
\end{itemize}
\end{lemma}

\begin{proof}
\textbf{Step 1.}
We first prove that 
(E3) $\Rightarrow$ (E4). 
Assume (E3) holds, such that $W(x)$ is a monotone profile satisfies~\eqref{BC3}. 
Let $\epsilon>0$, and 
consider the limit as $x\to +\infty$.
There exists an $M$ such that for all $x \ge M$ we have 
\[
 \rho_+ -\epsilon <  W(x) \le  \rho_+.
 \]
Since $\phi'(\cdot) <0$, we have, for all $x \ge M$, 
\[
\frac{1}{\phi(\rho_+ -\epsilon)}< \frac{1}{\phi(W(x))} \leq \frac{1}{\phi(\rho_+)}. 
\]
Integrating this inequality over $[x, x+\ell/W(x)]$, one has, for all $x \ge M$, 
\[
\int_{x}^{x+\ell/W(x)} \frac{1}{\phi(\rho_+ -\epsilon)} \, dz <
 \int_{x}^{x+\ell/W(x)}   \frac{1}{ \phi(W(z))}\, dz  \, \leq \int_{x}^{x+\ell/W(x)}  \frac{1}{\phi(\rho_+)} \, dz.
\]
This gives
\[
\frac{\ell}{\rho_+} \cdot \frac{1}{\phi(\rho_+ -\epsilon)}
< V \cdot t_p \le \frac{\ell}{\rho_+-\epsilon} \cdot \frac{1}{\phi(\rho_+)}.
\]
Taking the limit $\epsilon \to 0$, we get
\[
 t_p = \frac{\ell}{V \cdot \rho_+\phi(\rho_+)} = \frac{\ell}{f(\rho_+)}.
 \]
The other limit $x\to-\infty$ can be treated in a completely similar way,
proving (E4). 

\textbf{Step 2.} 
The implication (E4) $\Rightarrow$ (E3) follows by contradiction. 
Assuming that (E4) holds, but 
\[
\lim_{x\to\infty} W(x) = \hat \rho_{\pm}, \quad 
\hat\rho_- <\rho^*<\hat \rho_+, \quad 
f(\hat \rho_{\pm}) = \hat f \not=\bar f.
\]
By the proof in Step 1 we have the contradiction
$t_p =\ell/\hat f \not= \ell/\bar f$. 
\end{proof}

%%%%%%%%%%%%%%%%%
\section{Approximate Sequence; Existence and Uniqueness 
of Traveling Wave Profiles}
\setcounter{equation}{0}

We now construct approximate solutions to $W(x)$ 
as a two-point-boundary-value problem,
and prove their convergence, thus
establishing the existence of traveling wave profiles.

\begin{theorem} \label{tm1}
(Existence.)
Given $\ell, V, \rho_-,\rho_+$, with 
\begin{equation}\label{eq:tm1a}
0 \le \rho_-\le \rho^* \le \rho_+ \le 1, \qquad f(\rho_-)=f(\rho_+),
\end{equation}
there exists a monotone stationary  profile $W(x)$ which satisfies the DDE~\eqref{eq:W}
and the ``boundary'' values
\begin{equation}\label{eq:tm1b}
 \lim_{x\to-\infty} W(x) = \rho_-,\qquad 
 \lim_{x\to+\infty} W(x) = \rho_+.
 \end{equation}
\end{theorem}

\begin{proof} 
By Remarks~\ref{rk:2.2}-\ref{rk:2.4},  we rule out the trivial cases. 
For the rest of the proof, we consider
\begin{equation}\label{a1}
0 < \rho_- < \rho^* < \rho_+<1.
\end{equation}
The proof takes a few steps.

\textbf{(1).}
We first construct the sequence of approximate solutions. 
Let a sequence $\{\hat x_n\}$ be given such that 
\[
\hat x_n > 0, \qquad \hat x_n < \hat x_{n+1}, \qquad \lim_{n\to\infty} \hat x_n = \infty.
\]
We  define the function 
\[
\psi(x) =  \rho_+ -  e^{-\lambda_+ x},
\]
where $\lambda_+$ is the rate computed  in Lemma~\ref{lm:4}. 
Here we set the constant $M=-1$, since different values of it would only lead to 
a horizontal shift of the profile.  
Given $n$,  let $\psi(x)$ be the boundary condition for the DDE~\eqref{eq:W}
on $x\in[\hat x_n,\infty)$, 
and denote the corresponding solution as $W_n(x)$, for $x\le\hat x_n$.

From Theorem~\ref{IVP}, $W_n(x)$ is monotonically increasing and bounded below 
as $x\to-\infty$.  Denoting that 
\begin{equation}\label{Wlim}
\rho_{-,n} \, \dot= \, \lim_{x\to-\infty} W_n(x),
\end{equation}
it remains to show that 
\begin{equation}\label{rlim}
\lim_{n\to\infty} \rho_{-,n} = \rho_-.
\end{equation}

\textbf{(2).}
We derive an estimate for $\rho_{-,n}$.
Given an $n$ such that $\hat x_n$ is sufficiently large, so 
\[
\psi(x)  = \rho_+ -e^{-\lambda_+ x} \approx \rho_+,
\qquad \mbox{for}\quad x>\hat x_n.
\]
Consider the solution $W_n(x)$, given on $x\le \hat x_n$. 
Let $\{z_i^n\}$ be a distribution of car positions generated by
$W_n(x)$, with $z_0^n=\hat x_n$, for $i=0,-1,-2,-3,\cdots$.
Let $\{z_i^n(t)\}$ be the solution of the system of ODEs~\eqref{FtL},
for index $i<0$, and the leader $z_0^n(t)$ 
traces along the initial condition $\psi(x)$ on $x>\hat x_n$.
 
By Lemma~\ref{lm:2}, $\{z_i^n(t)\}$ ($i<0$) demonstrates periodic behavior. 
We denote the period by $t_{p,n}$.  
Note that once $t_{p,n}$ is given, we obtain the unique value of 
$\rho_{-,n}$ by the relation
\[
 \rho_{-,n} \le \rho^*, \qquad f(\rho_{-,n}) = \frac{\ell}{t_{p,n}}.
 \]
Thus, it suffices to show that 
\begin{equation}\label{tplim}
\lim_{n\to\infty} t_{p,n} = t_p =\frac{\ell}{f(\rho_\pm)}.
\end{equation}

We further observe that, thanks to the periodic behavior of $\{z_i^n(t)\}, (i<0)$,
we only need to get an estimate of the time for car located at $z_{-1}^n$,
to reach $z_0^n=\hat x_n$.  Here $z_{-1}^n$ is uniquely defined 
by the implicit relation
\[
 z_{-1}^n  + \frac{\ell}{W(z_{-1}^n)} = z_0^n.
 \]

\textbf{(3).} By the set up, $W_n(x)$ is very close to $\rho_+$
 on the interval $[z_{-1}^n, z_0^n]$, therefore an estimate on $t_{p,n}$
 can be obtained by linearization. 
 By Lemma~\ref{lm:4}, we have the first order approximation of $W_n(x)$
 on $[z_{-1}^n, z_0^n]$, denoted as 
 \[
  W_n(x) \approx  \rho_+ - e^{-\lambda_+ x},  \qquad x\in [z_{-1}^n, z_0^n].
  \]
 To simplify the notation, we denote the magnitude of the perturbation by 
 \begin{equation}\label{eq:delta}
 \delta_n \;\dot=\;  e^{-\lambda_+ z_0^n} = e^{-\lambda_+ \hat x_n}.
 \end{equation}
 The first order approximation for $W_n(x)$ can be written as
 \[
 W_n(x) = \rho_+ - \delta_n e^{-\lambda_+ (x-z_0^n)}. 
 \]
 
The corresponding distance $ z_0^n-z_{-1}^n$ is  computed approximately as
\begin{equation}\label{zz}
 z_0^n-z_{-1}^n 
 = \frac{\ell}{\rho_+ - \delta_n e^{\lambda_+ (z_0^n-z_{-1}^n)}} 
=\frac{\ell}{\rho_+} +\delta_n\cdot
 \frac{\ell}{\rho_+^2} e^{\lambda_+ \ell/\rho_+}
 + \mathcal{O}\left(\delta_n^2\right).
\end{equation}
 
 We can compute $t_{p,n}$, using a first order approximation in $\delta_n$, as 
 \begin{eqnarray*}
 t_{p,n} &=& \frac{1}{V} \int_{z_{-1}^n}^{z_0^n} \frac{1}{\phi(W(z))} dz  ~=~
  \frac{1}{V} \int_{z_{-1}^n}^{z_0^n} \frac{1}{\phi(\rho_+ - \delta_n 
e^{-\lambda_+ (z-z_0^n)})}dz\\ 
&=&  \frac{1}{V} \int_{z_{-1}^n}^{z_0^n} \frac{1}{\phi(\rho_+) - \delta_n 
e^{-\lambda_+ (z-z_0^n)} \phi'(\rho_+)}dz \\ 
&=& \frac{1}{V \cdot \phi(\rho_+)} \int_{z_{-1}^n}^{z_0^n} \left[1+ \delta_n 
e^{-\lambda_+ (z-z_0^n)} \frac{\phi'(\rho_+)}{\phi(\rho_+)}\right] \, dz\\
&=& \frac{1}{V \cdot \phi(\rho_+)} \left[(z_0^n-z_{-1}^n) + \delta_n \cdot 
\frac{\phi'(\rho_+) }{\lambda_+ \phi(\rho_+)}\left[e^{\lambda_+(z_0^n-z_{-1}^n)}-1\right]
\right]. 
 \end{eqnarray*}
 Using~\eqref{zz}, we get
 \begin{eqnarray}
 t_{p,n} &=& \frac{1}{V \cdot \phi(\rho_+)} \left[\frac{\ell}{\rho_+}  + \delta_n \cdot 
\left\{\frac{\ell }{\rho_+^2} e^{\lambda_+ \ell/\rho_+}+
\frac{\phi'(\rho_+) }{\lambda_+ \phi(\rho_+)}\left[e^{\lambda_+\ell/\rho_+}-1\right]
\right\}
\right] \nonumber\\[2mm]
&= &t_p + M \delta_n, \label{tpn}
 \end{eqnarray}
 where $M$ is a positive constant, depending on the data $\rho_+,\phi,V,\ell$, 
 but not on $\delta_n$.
As $n\to\infty$, $\delta_n \to 0$, and we conclude~\eqref{tplim}, 
completing the proof.
\end{proof}

In Figure~\ref{fig:SimE} we present some numerical simulations 
of the approximate solutions $W_n(x)$. 
We  obverse the convergence as $\hat x_n\to\infty$. 
Furthermore, from~\eqref{tpn} it holds that $t_{p,n} > t_p$, 
indicating $\rho_{-,n} < \rho_-$, which is consistent with the simulation results.

\begin{figure}[htbp]
\begin{center}
\includegraphics[width=10cm]{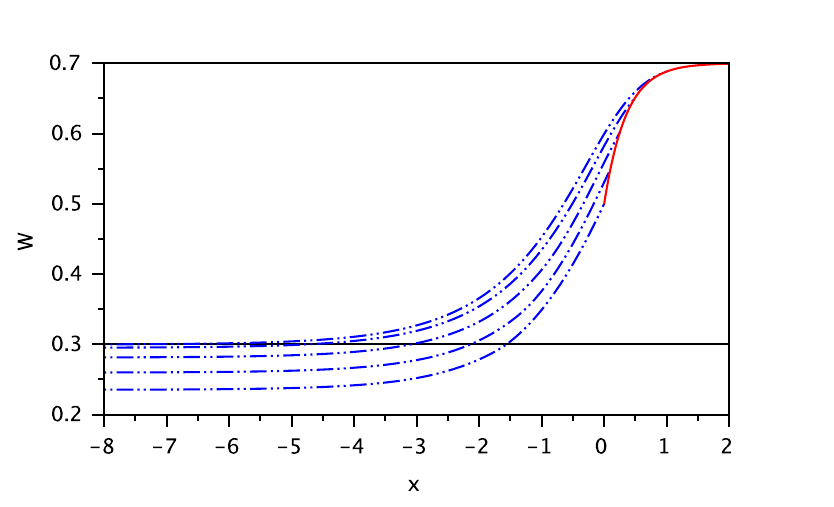}
\caption{Numerical simulations for the approximate sequence $W_n(x)$ for various
values of $\hat x_n$.  
We use $\rho_-=0.3$, $\rho_+=0.7$, $\ell=0.5$, $V=1$, and $\phi(\rho)=1-\rho$.
The solid curve is the graph of $\psi(x)=\rho_+-0.2 e^{-\lambda_+x}$, plotted on the 
interval $0\le x\le 2$. 
The dotted curves are plots for $W_n(x)$ with $\hat x_n=0, 0.1, 0.25, 0.5, 1$.}
\label{fig:SimE}
\end{center}
\end{figure}

Once the existence of the profile $W(x)$ is proved,
we  establish the uniqueness of the solution 
for the ``two-point-boundary-value-problem'' for the DDE~\eqref{eq:W}.

\begin{theorem} \label{tm:3}
(Uniqueness.)
Consider the settings of Theorem~\ref{tm1}. 
The solution $W(x)$ is unique up to a horizontal shift.
\end{theorem}

\begin{proof}
We first consider the trivial cases.  
If $\rho_- = \rho_+=\rho^*$, the only monotone graph is $W(x)\equiv \rho^*$.
If $\rho_- = 0$ and  $\rho_+ =1$, then $t_p \to \infty$ and nothing moves,
so the flux must be 0 everywhere. 
The only  monotone solution is a unit step function. 
In the rest of the proof, we assume 
\[
0<\rho_- <\rho^*<\rho_+ <1.
\]
We prove by contradiction.
Consider the settings of Theorem~\ref{tm1}, and let $W(x), \bar{W}(x)$ 
be two  solutions which are different.
Assume that, after some horizontal shift, the graphs of $W(x)$ and $\bar{W}(x)$ 
intersect at a point $\hat x$ so that
\[ W(\hat x) = \bar W(\hat x)\]
and 
\[ W(x) > \bar W(x) , \quad \mbox{for} \quad \hat x < x < \hat x + \ell/W(\hat x).\] 
Then, by the periodical property, we have
\[
t_p = \int_{\hat x}^{\hat x+\ell/W(\hat x)} \frac{1}{V \phi(W(z))} dz
> \int_{\hat x}^{\hat x+\ell/\bar W(\hat x)} \frac{1}{V \phi(\bar W(z))} dz =t_p,
\]
a contradiction.

This means that, if
the graphs of $W$ and $\bar W$ cross each other at $\hat x$, then they must cross each 
other at least one more time on $(\hat x, \hat x + \ell/W(\hat x))$. 
Thus, they must cross each other infinitely many times for $x\in\mathbb{R}$.

We now freeze the graph of $W(x)$, and shift the graph of $\bar W(x)$ to
the right until they touch each other only at $\bar x$ tangentially, 
such that 
\[
W(\bar x) =\bar W(\bar x), \qquad W(x) > \bar W(x) \qquad  \mbox{for} \quad \bar x <x< \bar x+ \ell/W(\bar x).
\]
Again, by periodicity, we get 
\[
t_p = \int_{\bar x}^{\bar x+\ell/W(\bar x)} \frac{1}{V \phi(W(z))} dz
> \int_{\bar x}^{\bar x+\ell/\bar W(\bar x)} \frac{1}{V \phi(\bar W(z))} dz =t_p,
\] 
a contradiction. 

We conclude that the graphs of $W$ and $\bar W$ 
either completely coincide or never cross each other. 
Then they must be horizontal shifts of each other, proving the uniqueness. 
\end{proof}

\paragraph{Numerical simulations.} 
Various profiles of $W(x)$ are plotted in Figure~\ref{fig:MW}, 
for various values of $\rho_\pm$. 
Here we use $\ell=0.1$ and 
$\phi(\rho)=1-\rho$, such that $\rho^*=0.5$.  
We plot the graphs of $W(x)$ that connect the following pairs of 
limit values of $(\rho_-,\rho_+)$ 
\[
(0.4, 0.6), \quad (0.3, 0.7),\quad (0.2, 0.8), \quad (0.1, 0.9). % , \quad (0.01,0.99).
\]
The profiles are simulated numerically, 
obtained as the limits of the approximate sequences described in Theorem~\ref{tm1}. 
The profiles are further shifted horizontally such that $W(0)=0.5$. 

\begin{figure}[htbp]
\begin{center}
\includegraphics[width=10cm,clip,trim=5mm 0mm 18mm 8mm]{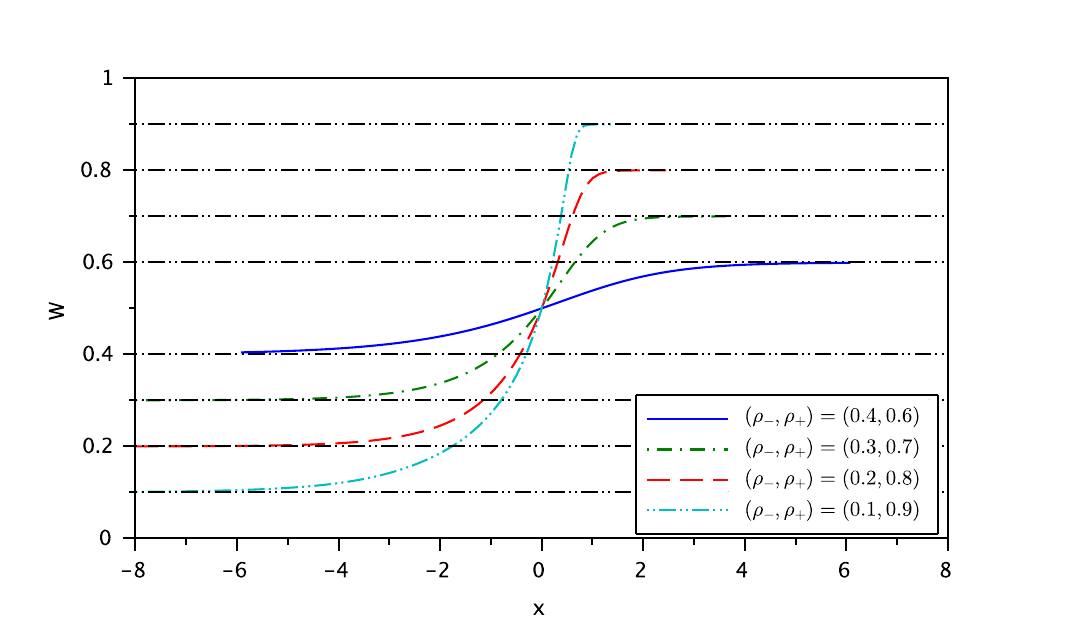}
\caption{Numerical simulations of stationary profiles $W(x)$ with various values of  
$(\rho_-,\rho_+)$.  }
% We observe the convergence to a step function as $(\rho_-,\rho_+)\to (0,1)$. }
\label{fig:MW}
\end{center}
\end{figure}

We make a couple of observations. 

(1).  For smaller values of $(\rho_+-\rho_-)$, the profiles $W(x)$ 
has a smaller value of $W'(0)$.  We provide a formal argument. 
For small $\ell$, $W(x)$ can be approximated by a linear function 
\[
 W(x) = W(0) + \sigma x, \qquad \sigma \;\dot=\; W'(0),
 \]
for $x$ close to $0$.  
Writing out only the first order approximations, we have
\[
 z_0 = W(0)=\rho^*=0.5, \qquad z_1= \frac{\ell}{W(0)} = \frac{\ell}{\rho^*}= 2\ell,
\qquad W(z_1) = \rho^* + \sigma \frac{\ell}{\rho^*} = 0.5+2\ell\sigma,
\]
and
\[
 f^* = f(\rho^*)= V/4, \qquad \bar f = f(\rho_\pm)=V\rho_\pm (1-\rho_\pm).
 \]
The periodic property gives
\[
  t_p = \frac{\ell}{\bar f} = \frac{1}{V}\int_0^{z_1} \frac{1}{1-W(x)} dx
= \frac{1}{V}\int_{\rho^*}^{\rho^*+\sigma\ell/\rho^*} \frac{dW}{\sigma(1-W)}.
\]
Working out the integration, we get
\[
\frac{V\ell}{\bar f} = -\frac{1}{\sigma} \ln \frac{1-\rho^*-\sigma \ell/\rho^*}{1-\rho^*}.
\]
Multiplying both sides by $-\sigma$ and then taking the exponential function on both sides, we obtain
\[
e^{-V\ell \sigma/\bar f} = \frac{1-\rho^*-\sigma \ell/\rho^*}{1-\rho^*} =1- \frac{\ell\sigma/\rho^*}{1-\rho^*} = 1- \frac{V\ell\sigma}{f^*} . 
\]
Moving everything to the right hand side, it gives the equation
\[
 K(\sigma) \;\dot=\;  1- \frac{V\ell\sigma}{f^*}  - e^{-V\ell \sigma/\bar f} =0.
 \]
Recall that $f^*=V/4$ is a constant.  The function $K(\sigma)$ has the following properties: 
\[
 K(0)=0, \qquad
K'(0) = V \ell (1/\bar f - 1/f^*) \ge 0, \qquad
K''(\sigma)= -(V\ell/\bar f)^2 e^{-V\ell \sigma/\bar f} <0.
\]
If $\bar f=f^*$, the only zero for $K(\sigma)$ is $\sigma=0$. 
When $\bar f < f^*$, then $K'(0) >0$, and there exists another positive
zero $\sigma_+$  for $K(\sigma)$.  
One can easily verify that $W'(0)=\sigma_+$ decreases as $\bar f$ increases 
to the value $f^*$, in correspondence to our simulation result.

(2). The profile $W(x)$ is not symmetric about $x=0$ in the 
sense that 
the asymptotic limit $\rho_+$ is approached much faster than
the limit $\rho_-$. 
Recall Lemma~\ref{lm:4}.  For this simulation, $\rho_-$ and $\rho_+$ 
locate symmetrically around $\rho^*=0.5$ such that $ \rho_- + \rho_+ =1$.
The estimates~\eqref{lambda+} and ~\eqref{lambda-2} give
\[
 \lambda_+ > \frac{2 \rho_+}{\ell} \ln(\rho_+/\rho_-), 
\qquad
\lambda_- < - \frac{2\rho_-}{\ell} \ln(\rho_-/\rho_+).
\]
Since $\rho_- < \rho_+$, we have $\lambda_- < \lambda_+$. 

We remark that this is different from a viscous shock profile, where the 
diffusion is uniform and the profile is 
odd symmetric about the location of the shock.

%%%%%%%%%%%%%%%
\section{Stability of the discrete traveling waves}
\setcounter{equation}{0}

We now show that the traveling wave profiles are local attractors for the solution of the FtL model.

\begin{theorem}\label{tm2} 
(Local Stability.) 
Let $W(x)$ be the unique stationary profile established in  Theorem~\ref{tm1}  with 
\[
W(0) = \rho^*, \qquad \mbox{where}\quad f'(\rho^*)=0.
\]
Let $\{z_i(t),\rho_i(t)\}$ be the solution of the ODEs~\eqref{FtL} with the initial data
$\{z_i(0), \rho_i(0)\}$.
Assume that there exist values $h_+,h_-$ with $h_+<h_-$, 
such that the initial data satisfies
\begin{equation}\label{eq:tm2.3}
W(z_i(0)-h_+)  \ge \rho_i(0) \ge W(z_i(0) - h_-), \qquad \forall i \in\mathbb{Z}. 
\end{equation}
Then, there exists a constant $\bar h$ such that 
\begin{equation}\label{eq:tm2.4}
\lim_{t\to\infty} W(z_i(t)-\bar h) - \rho_i(t) =0, \qquad \forall  i \in\mathbb{Z}. 
\end{equation}
\end{theorem}

Theorem~\ref{tm2} implies that 
as $t\to\infty$,  $\{z_i(t),\rho_i(t)\}$
approaches asymptotically a distribution generated by the profile $W(x-\bar h)$.

\begin{proof} 
Since $W(x)$ is monotonically increasing, with 
$\lim_{x\to\pm\infty} W(x) =\rho_\pm$, 
then for any point $(z,\rho)$ with $\rho_- < \rho<\rho_+$, there exists a 
unique value $h$ such that  $\rho = W(z-h)$.
We define the function
\[
 H(z,\rho) \;\dot=\; h, \qquad \mbox{where}\quad \rho = W(z-h) \quad \mbox{and}\quad 
 \rho_- < \rho <\rho_+.
\]
Then, 
for $t\ge 0$ and for each $i$, let 
\begin{equation}\label{eq:tm2.1}
h_i(t) = H(z_i(t),\rho_i(t)),\qquad \mbox{where} \quad \rho_i(t) =W(z_i(t)-h_i(t)).
\end{equation}

Denote 
\[
y_i(t) = z_i(t)-h_i(t),
\qquad
\mbox{so}\qquad
\rho_i(t) = W(y_i(t)).
\]
Differentiating in $t$, we get
\[
\dot \rho_i = W'(y_i) \dot y_i = W'(y_i) (\dot z_i - \dot h_i). 
\]
Using that 
\begin{eqnarray*}
\dot z_i &=& V \phi(\rho_i),\\
\dot \rho_i &=& \frac{V}{\ell}\rho_i^2\left(\phi(\rho_i) - \phi(\rho_{i+1})\right),\\
W'(y_i) &=& \frac{\rho_i^2}{\ell \phi(\rho_i)} \left[ \phi(\rho_i) - \phi(W(y_i+\ell/\rho_i))\right],
\end{eqnarray*}
we get
\begin{eqnarray*}
\dot h_i &=& \dot z_i -\frac{\dot \rho_i}{W'(y_i)} 
= V \phi(\rho_i) - V \phi(\rho_i) \frac{\phi(\rho_i) - \phi(\rho_{i+1})}{\phi(\rho_i) - \phi(W(y_i+\ell/\rho_i))}\\
&=& \frac{V \phi(\rho_i) }{ \left[ \phi(W(y_i)) - \phi(W(y_i+\ell/\rho_i))\right] }\cdot 
\left[  \phi(\rho_{i+1})-  \phi(W(y_i+\ell/\rho_i))\right]. 
\end{eqnarray*}
Then, if $h_{i}(t) < h_{i+1}(t)$, we have
\[W(y_i+\ell/\rho_i) > \rho_{i+1}.\]
Since $\phi$  is a monotonically decreasing function, we have
\[ \phi(W(y_i+\ell/\rho_i))<\phi(\rho_{i+1}) \quad\mbox{and}\quad
\phi(W(y_i)) > \phi(W(y_i+\ell/\rho_i))
,\]
which implies that $ \dot h_i (t)>0$. 
Similarly, if $h_{i}(t) > h_{i+1}(t)$ then we have $ \dot h_i(t) <0$.

Now we define
\begin{equation}\label{eq:tm2.2}
h^\sharp(t) \;\dot=\; \min_{i} h_i(t), \qquad 
h^\flat(t) \;\dot=\; \max_{i} h_i(t).
\end{equation}

It suffices to show that 
\begin{equation}\label{eq:tm2.5}
\lim_{t\to\infty}  \left[h^\flat (t)- h^\sharp(t)\right] =0.
\end{equation}

Indeed, for any given $t\ge 0$, from the previous discussion we have the following.
\begin{itemize}
\item Let $h_j(t)$ be a maximum.  If $h_j(t) >h_{j+1}(t)$, then $\dot h_j <0$; If 
$h_j(t) =h_{j+1}(t)$, then $\dot h_j =0$ and $h_{j+1}(t)$ is also a maximum;
\item Let $h_k(t)$ be a minimum.  If	 $h_k(t) <h_{k+1}(t)$, then $\dot h_k >0$;
If $h_k(t) =h_{k+1}(t)$, then $\dot h_k =0$ and $h_{k+1}(t)$ is also a minimum.
\end{itemize}
Then, 
\[
\frac{d}{dt} \left[h^\flat (t)- h^\sharp(t)\right] \le 0, 
\]
and the limit 
\[
\lim_{t\to\infty}  \left[h^\flat (t)- h^\sharp(t)\right] 
\] 
exists and is non-negative. 
To show that the limit must be 0, we use contradiction and assume the opposite, such that
\[
\lim_{t\to\infty}  \left[h^\flat (t)- h^\sharp(t)\right] = d_h>0,
\]
and let $\{\tilde z_i,\tilde\rho_i\}$ be the asymptotic  car distribution, with the corresponding values of
$\{\tilde h_i\}$. 
Now, take $\{\tilde z_i,\tilde\rho_i\}$ as the initial data and solve the system of 
ODEs~\eqref{FtL}.  
There exists an index $j$ where $\tilde h_j$ is the maximum with 
$\tilde h_j > \tilde h_{j+1}$. 
By the previous discussion we have 
$\frac{d}{dt}\tilde h_j <0$. 
If this is the isolated maximum, then $\frac{d}{dt} h^\flat <0$, a contradiction. 
If $\tilde h_{j-1}$ is also a maximum, then after an arbitrarily small amount of time 
we have $\frac{d}{dt} \tilde h_{j-1} <0$ so $\frac{d}{dt} h^\flat <0$, still a contradiction. 
Thus, we conclude~\eqref{eq:tm2.5}, completing the proof. 
\end{proof}

%%%%%%%%%%%%%%%
\section{Extension to general traveling waves}

One can extend the analysis to traveling waves with speed different from $0$, by a simple 
coordinate shift.  
Let $V\sigma$ be the wave speed and 
let $\xi = x-V \sigma t$ be the shifted space coordinate.
Let $\zeta_i(t)$ be the position of the $i$th car in the shifted coordinate, 
and $\rho_i$ the discrete density.
We have 
\[ \dot \zeta_i = \dot z -V\sigma = V\phi(\rho_i)-V\sigma  = V(\phi(\rho_i) -\sigma).\]
Since the density is not affected by a horizontal shift,
the ODE for $\rho_i$ is unchanged. 

Consider  a traveling wave profile $\mathcal{W}(\xi) = W(x-V\sigma t)$. 
We must have
\[ 
\mathcal{W}(\zeta_i(t)) = \rho_i(t), \qquad \forall t>0.
\]
This leads to the DDE:
\[
\mathcal{W}'(\xi) = \frac{\mathcal{W}^2(\xi)}{\ell(\phi(\mathcal{W}(\xi))-\sigma)} \Big[ \phi(\mathcal{W}(\xi)) - \phi(\mathcal{W}(\xi+\ell/\mathcal{W}(\xi)))\Big].
\]
The corresponding conservation law is
\[
 \rho_t + f(V,\rho)_\xi =0, \qquad \mbox{where} \quad 
f(V,\rho) = V \rho(1-\rho-\sigma).
\]
The analysis for the stationary traveling wave can be applied here with minimal
modifications.

%%%%%%%%%%%%%%%%%%%%%%
\section{Concluding Remark}

In this paper we study traveling wave profiles of  a particle model 
for traffic flow, i.e., the follow-the-leader (FtL) ODE models for car positions.
Given any densities  $\rho_\pm$ at $x\to \pm\infty$, 
with $\rho_- < \rho_+$,
we prove that there exists a unique traveling  wave profile for the FtL model.
Furthermore, such profiles are locally stable which attract nearby solutions
of the FtL model.
In the limit as  $\ell\to 0$, 
the traveling waves converge to admissible shocks for the
solution of the conservation law~\eqref{PDE}.
Our results fill a gap  in existing theory on traveling waves.
The admissible conditions
derived from our result
are in accordance to the counter part for the viscous equation
\[
 \rho_t + f(\rho)_x = \ve \rho_{xx},
 \]
where stable viscous shocks only exist for upward jumps. 

It's  interesting and also non-trivial to study   
the same particle model on a road with rough conditions. 
For example, let $\kappa(x)$ denote  the speed limit (which reflects the road condition),
and assume that it is a piecewise constant function with a jump at $x=0$.  
One would like to seek stationary traveling waves for the FtL model,
around $x=0$. 
The corresponding macroscopic model
\[
\rho_t + f(\rho,\kappa(x))_x =0
\]
 is a scalar conservation law with discontinuous flux. 
 In existing literature, 
admissibility conditions on the jump at $x=0$ are derived 
through the viscous model
\[
 \rho_t + f(\rho,\kappa(x))_x =\ve \rho_{xx}
 \]
and take the vanishing viscosity limit $\ve\to 0+$. 
However, our preliminary analysis shows a rather different scenario
for limits of the FtL model, as $\ell \to 0$, where many of the 
vanishing viscosity limits are actually not admissible.
Details are in a forthcoming work~\cite{ShenDDDE2017}.

\section*{Acknowledgement} 
The authors would like to thank the anonymous reviewer for the careful reading of the 
manuscript and useful comments which led to an improved version of this paper.

\bibsection 
%\end{thebibliography}

\end{document}